\documentclass[a4paper]{amsart}

\usepackage[american]{babel}
\usepackage[T1]{fontenc}
\usepackage[utf8]{inputenc}
\usepackage{amsmath,amssymb,amsfonts,amsthm,amstext}
\usepackage{dsfont,mathtools,enumerate,mathrsfs}
\usepackage{microtype}
\usepackage[numbers,sort&compress]{natbib}

\newcommand{\apply}[3][]{\left<#2,#3\right>_{#1}}

\renewcommand{\phi}{\varphi}

\newcommand{\eps}{\varepsilon}
\renewcommand{\rho}{\varrho}

\DeclareMathOperator{\divi}{div}

\theoremstyle{definition}
\newtheorem{theorem}{Theorem}[section]
\newtheorem{proposition}[theorem]{Proposition}

\newtheorem{lemma}[theorem]{Lemma}
\newtheorem{remark}[theorem]{Remark}
\newtheorem{example}[theorem]{Example}
\newtheorem{definition}[theorem]{Definition}

\title[Quasi-linear eigenvalue problems]{A note on the implicit function theorem for quasi-linear eigenvalue problems}
\date{September 23, 2011}
\author{Robin Nittka}
\address{Robin Nittka\\Max Planck Institute for Mathematics in the Sciences\\Inselstr. 22\\04103 Leipzig\\Germany}
\email{nittka@mis.mpg.de}

\numberwithin{equation}{section}

\begin{document}
\begin{abstract}
	We consider the quasi-linear eigenvalue problem
	$-\Delta_p u = \lambda g(u)$ subject to Dirichlet boundary conditions on a bounded open set $\Omega$, where
	$g$ is a locally Lipschitz continuous functions. Imposing no further conditions on $\Omega$ or $g$
	we show that for small $\lambda$ the problem has a bounded solution which is
	unique in the class of all small solutions. Moreover, this curve of solutions
	depends continuously on $\lambda$.
\end{abstract}
\maketitle

\section{Introduction}

We give an argument in the spirit of~\cite[\S3]{CR73} in order to prove existence of small solutions
for the quasi-linear equation
\begin{equation}\label{eq:scalar}
	\left\{ \begin{aligned}
		-\Delta_p u & = \lambda g(u) && \text{on } \Omega \\
		u & = 0 && \text{on } \partial\Omega, \\
	\end{aligned} \right.
\end{equation}
where $g\colon \mathds{R} \to \mathds{R}$ is locally Lipschitz continuous, $\lambda$ is small,
$\Omega$ is an arbitrary bounded open subset of $\mathds{R}^N$ and $p$ is in $(\frac{2N}{N+2}, 2]$.
More precisely, we show that there is a neighborhood $U$ of zero in $L^\infty(\Omega)$ such
that for small $\lambda$ there is a unique solution $u_\lambda \in U \cap W^{1,p}_0(\Omega)$ of~\eqref{eq:scalar}.
Moreover, the dependence of $u_\lambda$ on $\lambda$ is Lipschitz continuous in the norms
of $L^\infty(\Omega)$ and $W^{1,p}_0(\Omega)$.

Equation~\eqref{eq:scalar} is an example of a nonlinear eigenvalue problem (semilinear if $p=2$
and quasi-linear otherwise).
If $g$ has subcritical growth, the situation is comparatively easy to handle. Namely, if
$g(x) \le (1 + |x|)^q$ for all $x \in \mathds{R}$ for some $q < \frac{Np}{N-p}-1$,
then $u \mapsto g(u)$ is compact from $W^{1,p}_0(\Omega)$ to $W^{-1,p'}(\Omega)$
and~\eqref{eq:scalar} can for example be attacked by fixed point arguments in the space $W^{1,p}_0(\Omega)$
of energy solutions. The critical case $q = \frac{Np}{N-p}-1$ is more difficult because $u \mapsto g(u)$
is no longer compact, but it is still possible to work in $W^{1,p}_0(\Omega)$.
The subcritical and the critical case have been extensively studied,
usually with variational methods under additional monotonicity assumptions.
We refer to~\cite{Amann76} for a detailed survey on semilinear eigenvalue problem
and cite~\cite{DH01,KL03,AMM05,CS07,ILU10} as examples for results in the quasi-linear case.

Here we also want to allow for supercritical growth.
We think of $g = \exp$ as a model case. With this choice for $g$ equation~\eqref{eq:scalar} is often referred to as
Gelfand's equation, in particular if $p=2$, and has a physical interpretation
in astrophysics~\cite{Ch57}.
For $p=2$ the situation is quite well understood, see for example~\cite{Dancer88,CR75,AD07,PM02}
for information about the number of solutions.
For $p \neq 2$ less is known, but we refer to~\cite{AAP94} and references therein
for results about existence, non-existence, multiplicity and stability.
However, the methods of~\cite{AAP94} rely heavily on structure assumptions on $g$,
namely some kind of monotonicity and growth conditions, and thus do not generalize well.
Here we propose a different approach to~\eqref{eq:scalar} based on the implicit
function theorem, which goes along the lines of~\cite{CR75} where the case $p=2$ is studied.
We emphasize the method's flexibility by allowing for general functions $g$
and by later on studying systems of $p$-Laplace equations instead of scalar equations.

There are two major difficulties to be overcome when using the implicit function theorem for~\eqref{eq:scalar}.
Firstly, one has to work in a space of bounded functions
in order to be able to handle composition with $g$. For arbitrary domains
a space of H\"older continuous functions like in~\cite{CR75} is not suitable,
so we have to resort to $L^\infty(\Omega)$ like in~\cite{Dipl}.
Secondly, equation~\eqref{eq:scalar} does not behave well under linearization unless $p=2$,
so we cannot expect that the implicit function theorem for continuously differentiable
functions applies and have to use a topological version for compact Lipschitz maps.
Thus we need compactness of the resolvent of the $p$-Laplace operator in $L^\infty(\Omega)$,
which for rough domains seems to be a new result,
and we need the local Lipschitz continuity of $(-\Delta_p)^{-1}$ in $L^\infty(\Omega)$,
which has recently been obtained by Markus Biegert~\cite{Bie10}.
Local Lipschitz continuity fails for $p > 2$ and is not known for $p \in (1,\frac{2N}{N+2}]$,
which is the main reason why we have to restrict ourselves to $p \in (\frac{2N}{N+2},2]$.

\section{$\Delta_p$ and local Lipschitz continuity of the resolvent}
Throughout the article $\Omega$ denotes a fixed bounded open subset of $\mathds{R}^N$
and $p$ is a parameter satisfying $p > \frac{2N}{N+2}$. Thus $W^{1,p}_0(\Omega)$
is compactly embedded into $L^2(\Omega)$. Later on we will also
require that $p \le 2$. We define
\begin{equation}\label{eq:phip}
	\phi_p(u) \coloneqq \frac{1}{p} \int_\Omega |\nabla u|^p
\end{equation}
for $u \in W^{1,p}_0(\Omega) \subset L^2(\Omega)$ and $\phi_p(u) \coloneqq \infty$
for $u \in L^2(\Omega) \setminus W^{1,p}_0(\Omega)$. Then $\phi_p$ is a proper, convex,
lower semicontinuous functional on $L^2(\Omega)$. Its subgradient $-\Delta_p$ can
be described by $u \in D(\Delta_p)$ and $-\Delta_p u = f \in W^{-1,p'}(\Omega)$ if and only if
$u \in W^{1,p}_0(\Omega)$ and
\[
	\int_\Omega |\nabla u|^{p-2} \nabla u \; \nabla v = \apply{f}{v} \text{ for all } v \in W^{1,p}_0(\Omega),
\]
which means that on a formal level we can write $\Delta_p u = \divi( |\nabla u|^{p-2} \nabla u)$. The operator
$I - \alpha \Delta_p\colon W^{1,p}_0(\Omega) \to L^2(\Omega)$
is invertible for all $\alpha > 0$. Its inverse $J_\alpha \coloneqq (I + \alpha \Delta_p)^{-1}$ is contractive
with respect to the norm of $L^2(\Omega)$.
Moreover, the resolvent identity
\begin{equation}\label{eq:pseudo_resolvent}
	J_\beta = J_\alpha\Bigl( \frac{\alpha}{\beta} I + \bigl(1-\frac{\alpha}{\beta}\bigr) J_\beta \Bigr)
\end{equation}
holds for all $\alpha, \beta > 0$. The operator $-\Delta_p$ is coercive on $L^2(\Omega)$
and hence invertible due to Poincar\'e's inequality.
For proofs of these facts and further details we refer to~\cite{Showalter}.

We state the local Lipschitz continuity of $(-\Delta_p)^{-1}$ in the following lemma.
This limit case is not included in~\cite[Theorem~5.1]{Bie10}, but can be obtained
from~\cite[Theorem~3.1]{Bie10} by precisely the same arguments.

\begin{lemma}\label{lem:lipschitz}
	Let $p \in (\frac{2N}{N+2},2]$. If we pick $q \in (1,\infty)$ sufficiently large,
	we have the following property: if $u, v \in W^{1,p}_0(\Omega)$
	satisfy $-\Delta_p u = f$ and $-\Delta_p v = g$ with
	functions $f, g \in L^q(\Omega)$. Then
	\begin{equation}\label{eq:Best}
		\|u - v\|_\infty \le c \bigl( \|f\|_q + \|g\|_q \bigr)^{\frac{2-p}{p-1}} \|f-g\|_q
	\end{equation}
	for a constant $c$ that depends only on $p$ and $\Omega$.
\end{lemma}
%\begin{proof}
%	For $t > 0$ let $u(t)$ and $v(t)$ be the unique functions in $W^{1,p}_0(\Omega)$ that satisfy
%	\[
%		t |u(t)|^{p-2} u(t) - \Delta_p u(t) = f
%		\qquad\text{and}\qquad
%		t |v(t)|^{p-2} v(t) - \Delta_p v(t) = g.
%	\]
%	From~\cite[Theorem~5.1]{Bie10} we obtain that
%	\begin{equation}\label{eq:lest}
%		\|u(t) - v(t)\|_\infty \le c_p \bigl( \|f\|_q + \|g\|_q \bigr)^{\frac{2-p}{p-1}} \|f-g\|_q
%	\end{equation}
%	with $c$ depending only on $t > 0$.
%
%	The functions $u(t)$ converge to $u(0) = u$ weakly in $W^{1,p}_0(\Omega)$ as $t \to 0$.
%	This follows for example from~\cite[Proposition~5.7 and Corollary~7.24]{Maso} once
%	we note that $u(t)$ is the unique minimizer of $\phi_{t,f}$ defined by
%	\[
%		\phi_{t,f}(w) \coloneqq t \int_\Omega |w|^p + \int_\Omega |\nabla w|^p - \int_\Omega fw
%	\]
%	for all $w \in W^{1,p}_0(\Omega)$, see the proof of~\cite[Proposition~IV.1.5]{Showalter}.
%	Hence $u(t) \to u$ in $L^p(\Omega)$, and analogously $v(t) \to v$ in $L^p(\Omega)$.
%	By lower semicontinuity of the $L^\infty$-norm in $L^p(\Omega)$ we can pass in~\eqref{eq:lest}
%	to the limit and obtain~\eqref{eq:Best}.
%\end{proof}

Lemma~\ref{lem:lipschitz} says in particular that $(-\Delta_p)^{-1}$ is locally Lipschitz
continuous on $L^\infty(\Omega)$. This seems to be a non-trivial result. For example,
the scaling behavior
\[
	(-\Delta_p)^{-1} (\lambda f) = \lambda^{\frac{1}{p-1}} (-\Delta_p)^{-1}f.
\]
implies that $(-\Delta_p)^{-1}$ cannot be Lipschitz continuous in
any neighborhood of zero if $p > 2$, and local Lipschitz
continuity for $p \in (1, \frac{2N}{N+2}]$ seems to be an open problem.

On the other hand, global Lipschitz continuity of $J_\alpha \coloneqq (I - \alpha \Delta_p)^{-1}$
on $L^\infty(\Omega)$ for every $\alpha > 0$ and every $p \in (1,\infty)$ is comparatively easy to establish.
\begin{lemma}\label{lem:contr}
	For every $q \in [2,\infty]$ the restriction of $J_\alpha$ to $L^q(\Omega)$ is contractive
	in the norm of $L^q(\Omega)$.
\end{lemma}
\begin{proof}
	It can be proved as in~\cite[Theorem~4.1]{CG03} that $\phi_p$ defined in~\eqref{eq:phip}
	is a nonlinear Dirichlet form in the sense of~\cite{CG03}, i.e., the corresponding semigroup on $L^2(\Omega)$
	is order preserving and contractive in the norm of $L^\infty(\Omega)$. Thus the resolvent
	has the same properties~\cite{Brezis}, and by interpolation we obtain the result for all $q \in [2,\infty]$,
	see for example~\cite{Br69}.
\end{proof}

\section{Compact resolvent}
We prove that $J_\alpha \coloneqq (I - \alpha \Delta_p)^{-1}$
acts as a compact operator on $L^\infty(\Omega)$. Recall that
a nonlinear operator is called compact if it is continuous and maps bounded sets to relatively
compact sets. An operator is called bounded if it maps bounded sets to bounded sets, so trivially every compact
operator is bounded.
Compactness of $J_\alpha$ (and similar operators) in $L^q$-spaces for $q < \infty$ has been studied
in~\cite[\S5]{DD09}. However, the case $q=\infty$ seems to be new and is in fact not accessible
by the methods in~\cite{DD09}. To be more precise, the argument in~\cite{DD09} goes via an interpolation inequality,
so in order to transfer their idea to $L^\infty(\Omega)$ one needs to prove boundedness of the resolvent
into in a space of higher regularity. This regularity, however, is only to be expected if
$\Omega$ is sufficiently smooth. Thus instead we use the following different approach, adapted
from~\cite[Theorem~7.1]{AD07}, which we present in an abstract framework.

\begin{definition}
	Let $X$ be a Banach space. We say that a family
	$(J_\alpha)_{\alpha > 0}$ of (possibly nonlinear) operators on $X$
	is a (nonlinear) \emph{pseudo-resolvent} if~\eqref{eq:pseudo_resolvent}
	holds for all $\alpha, \beta > 0$.
\end{definition}

\begin{lemma}\label{lem:operatornorm_conv}
	Let $(J_\alpha)_{\alpha > 0}$ be a pseudo-resolvent on a Banach space $X$. Assume
	that the family $(J_\alpha)_{\alpha > 0}$ is uniformly bounded and that $(J_\alpha)_{\alpha > 0}$
	is equicontinuous on bounded subsets of $X$.
	Then for all $k \in \mathds{N}_0$ and all $\beta > 0$ we have
	$J_\alpha^k J_\beta \to J_\beta$ as $\alpha \to 0$
	uniformly on bounded sets.
\end{lemma}
\begin{proof}
	Let $\beta > 0$ be fixed. We prove the claim by induction.
	For $k = 0$ there is nothing to show.
	So assume that $J_\alpha^k J_\beta \to J_\beta$ as $\alpha \to 0$ uniformly on
	bounded sets for some $k \in \mathds{N}$. By~\eqref{eq:pseudo_resolvent},
	\[
		J_\alpha^k J_\beta
			= J_\alpha^{k+1} \Bigl( \frac{\alpha}{\beta} I + \bigl( 1 - \frac{\alpha}{\beta} \bigr) J_\beta \Bigr)
			= J_\alpha^{k+1} J_\beta + T_\alpha
	\]
	with
	\[
		T_\alpha \coloneqq J_\alpha^{k+1} \Bigl( \frac{\alpha}{\beta} I + \bigl( 1 - \frac{\alpha}{\beta} \bigr) J_\beta \Bigr)
				- J_\alpha^{k+1} J_\beta.
	\]
	Hence it suffices to show that $T_\alpha \to 0$ uniformly
	on bounded sets. To this end let $(x_\alpha)_{\alpha > 0}$ be a bounded family in $X$.
	Since $J_\beta$ is bounded, there exists $R > 0$ such that
	the vectors
	\[
		y_\alpha \coloneqq \frac{\alpha}{\beta} x_\alpha + \bigl(1 - \frac{\alpha}{\beta}\bigr) J_\beta x_\alpha 
		\qquad\text{and}\qquad
		z_\alpha \coloneqq J_\beta x_\alpha
	\]
	lie in the ball $B(0,R)$ for every $\alpha < 1$. In particular,
	\[
		\| y_\alpha - z_\alpha \|
			= \frac{\alpha}{\beta} \| x_\alpha - J_\beta x_\alpha \|
			\to 0
			\quad (\alpha \to 0).
	\]
	By uniform equicontinuity this implies that $\|J_\alpha y_\alpha - J_\alpha z_\alpha\| \to 0$
	as $\alpha \to 0$.
	Since also the families $(J_\alpha y_\alpha)_{\alpha > 0}$ and $(J_\alpha z_\alpha)_{\alpha > 0}$
	are bounded, we can proceed by induction and deduce that
	\[
		\|T_\alpha x_\alpha \| = \| J_\alpha^{k+1} y_\alpha - J_\alpha^{k+1} z_\alpha \| \to 0 \quad (\alpha \to 0).
	\]
	Thus $T_\alpha \to 0$ as $\alpha \to 0$ uniformly on bounded sets.
\end{proof}

\begin{theorem}\label{thm:compact_resolvent}
	Let $X$ and $Y$ be a Banach spaces with $Y$ continuously
	embedded into $X$. Let $(J_\alpha)_{\alpha > 0}$ be a pseudo-resolvent on $X$ consisting
	of continuous, bounded operators.
	Assume that $J_\alpha$ is compact on $X$ for one (or, equivalently, for all) $\alpha > 0$.
	Assume moreover that $J_\alpha(Y) \subset Y$ for all $\alpha > 0$ and that
	$(J_\alpha|_Y)_{\alpha > 0}$ is uniformly bounded and uniformly equicontinuous on bounded subsets of $Y$.
	Finally, assume that there exist $k \in \mathds{N}$ and $\beta > 0$ such that $J_\alpha^k$ is
	bounded and continuous from $X$ into $Y$ for all $\alpha > 0$.
	Then $J_\alpha|_Y$ is compact on $Y$ for every $\alpha > 0$.
\end{theorem}
\begin{proof}
	Since $J_\alpha^k$ is bounded from $X$ to $Y$
	and $J_\beta$ is compact on $X$, the operator $J_\alpha^k J_\beta$ is a compact operator
	from $X$ to $Y$ for all $\alpha > 0$ and $\beta > 0$. Hence $J_\alpha^k J_\beta|_Y$
	is a compact operator on $Y$.
	From Lemma~\ref{lem:operatornorm_conv} we obtain that $J_\alpha^k J_\beta|_Y \to J_\beta|_Y$
	as $\alpha \to 0$ uniformly on bounded sets. Hence $J_\beta|_Y$ is compact.
\end{proof}

We need Theorem~\ref{thm:compact_resolvent} only in the following situation.
\begin{example}\label{ex:pLaplace}
	Let $\Omega$ be an open, bounded subset of $\mathds{R}^N$. For every $p \in (\frac{2N}{N+2},\infty)$
	the operator $-\Delta_p$ has compact resolvent on $L^\infty(\Omega)$.
	Moreover, also the operator $(-\Delta_p)^{-1}$ is compact on $L^\infty(\Omega)$.
\end{example}
\begin{proof}
	Define $J_\alpha \coloneqq (I - \alpha \Delta_p)^{-1}$.
	Since $u \mapsto \phi_p(J_\alpha u)$ is bounded on bounded subsets of $L^2(\Omega)$,
	see for example~\cite[Proposition~IV.1.8]{Showalter}, the operator $J_\alpha$
	is bounded from $L^2(\Omega)$ to $W^{1,p}_0(\Omega)$. Since in addition $J_\alpha$ is continuous
	(in fact even contractive) on $L^2(\Omega)$, it is compact on $L^2(\Omega)$.

	The restriction of $J_\alpha$ to $L^\infty(\Omega)$ is equicontinuous and uniformly
	bounded by Lemma~\ref{lem:contr}.
	Iterating the estimates in~\cite[Theorem~2.5]{DD09} we obtain that
	$J_\alpha^k$ is bounded from $L^2(\Omega)$ to $L^\infty(\Omega)$ for some $k \in \mathds{N}$.
	Hence $J_\alpha^k$ is continuous from $L^2(\Omega)$ to $L^q(\Omega)$ for every $q < \infty$
	by interpolation. Thus $J_\alpha^{k+1}$ is bounded and continuous from $L^2(\Omega)$
	to $L^\infty(\Omega)$ by Lemma~\ref{lem:lipschitz}.
	Compactness of $J_\alpha$ on $L^\infty(\Omega)$ now follows from Theorem~\ref{thm:compact_resolvent}.

	Finally, $(-\Delta_p)^{-1}$ is bounded and continuous on $L^\infty(\Omega)$ by Lemma~\ref{lem:lipschitz}.
	Hence the trivial identity
	\[
		(-\Delta_p)^{-1} = J_\alpha\bigl( (-\Delta_p)^{-1} + \alpha I \bigr),
	\]
	which is valid for all $\alpha > 0$, shows that $(-\Delta_p)^{-1}$ is compact
	on $L^\infty(\Omega)$.
\end{proof}

\section{A topological implicit function theorem}
We need an implicit function theorem in Banach spaces for functions that are merely
Lipschitz continuous. The following argument is based on a rather deep theorem
about local inverses that relies on topological degree theory.
For finite-dimensional spaces this idea has been described in~\cite{Wuertz}.
The generalization to infinite dimensions is straight-forward, but
seems not to be widely known.

\begin{proposition}\label{prop:curve}
	Let $X$, $Y$ and $Z$ be Banach spaces.
	Let $K\colon X \times Y \to Z$ be a locally Lipschitz continuous compact operator.
	Assume that $K(0,0) = 0$ and that there exist $\kappa < 1$ and $\delta > 0$ such that
	$\|K(x,y_1) - K(x,y_2)\| \le \kappa \|y_1 - y_2\|$ whenever $x \in B_X(0,\delta)$
	and $y_1, y_2 \in B_Y(0,\delta)$.
	Then there exists a neighborhood $U \times V$ of $(0,0)$ in $X \times Y$
	and a Lipschitz continuous function $\phi\colon U \to V$ with the property
	that for $(x,y) \in U \times V$ we have $K(x,y) = y$ if and only if $y = \phi(x)$.
\end{proposition}
\begin{proof}
	Let $L \ge 0$ denote the Lipschitz constant of $K$ in a neighborhood of $(0,0)$
	and pick $0 < \eps < L^{-1}$.
	Setting $f(x,y) \coloneqq (x, \eps y - \eps K(x,y))$ we easily obtain that
	\begin{equation}\label{eq:fest}
		\|f(x_1,y_1) - f(x_2,y_2)\|
%				& \qquad \ge \|x_1 - x_2\| + \eps \|y_1 - y_2\|
%					- \eps \|K(x_1,y_1) - K(x_1,y_2)\|
%					- \eps \|K(x_1,y_2) - K(x_2,y_2)\| \\
			\ge (1 - \eps L) \|x_1 - x_2\| + \eps (1-\kappa) \|y_1 - y_2\|,
	\end{equation}
	from which we see that $f$ is injective near $(0,0)$. Since $f$ is a compact
	perturbation of the identity this implies
	that $f$ is continuously invertible in a neighborhood of $(0,0)$, see~\cite[(5.4.11)]{Berger}.
	Define $\phi(x)$ to be the second component of $f^{-1}(x,0)$.
	By~\eqref{eq:fest} the operator $f^{-1}$ is Lipschitz continuous, hence so is $\phi$.
	Moreover, $K(x,y) = y$ if and only if $f(x,y) = (x,0)$, which for $(x,y)$ in a neighborhood
	of $(0,0)$ is equivalent to $\phi(x) = y$.
\end{proof}

\section{Existence of small solutions}
We now prove existence of small solutions of~\eqref{eq:scalar} (or the more general equation~\eqref{eq:system})
by constructing a curve of solutions emanating from $(\lambda,u) = (0,0)$, for which we use the implicit function theorem
of the previous section.
In order to emphasize the flexibility of our approach, in particular that the method does not
rely on the variational structure of~\eqref{eq:scalar},
we consider a quasi-linear reaction-diffusion system instead of a scalar equation,
compare also~\cite{AC02}.

\begin{theorem}\label{thm:curve}
	Let $\Omega \subset \mathds{R}^N$ be open and bounded, let $d \in \mathds{N}$,
	fix numbers $p_1,\dots,p_d \in (\frac{2N}{N+2},2]$ and
	let $g\colon \mathds{R}^d \to \mathds{R}^d$ be locally Lipschitz continuous. Then there exist $\lambda_0 > 0$
	such that for every $\lambda \in (-\lambda_0,\lambda_0)$ the system
	\begin{equation}\label{eq:system}
		\left\{ \begin{aligned}
			-\Delta_{p_i} u_i & = \lambda g_i(u_1,\dots,u_d) && \text{on } \Omega \\
			u_i & = 0 && \text{on } \partial\Omega \\
		\end{aligned} \right.
	\end{equation}
	has a solution $(u_1,\dots,u_d) = u = u_\lambda$ in $W^{1,p}_0(\Omega; \mathds{R}^d) \cap L^\infty(\Omega; \mathds{R}^d)$;
	in particular $g(u_\lambda) \in L^\infty(\Omega;\mathds{R}^d) \subset W^{-1,p'}(\Omega;\mathds{R}^d)$
	so that the notion of an (energy) solution applies here.

	Moreover, the mapping $\lambda \to u_\lambda$ is Lipschitz continuous from $(-\lambda_0,\lambda_0)$
	to $L^\infty(\Omega) \cap W^{1,p}_0(\Omega)$.
	Finally, there exist $\eps > 0$ and $\lambda_1 \in (0,\lambda_0)$
	such that for all $\lambda \in (-\lambda_1,\lambda_1)$
	the function $u_\lambda$ is the only solution $u$ of~\eqref{eq:system} that
	satisfies $\|u\|_\infty \le \eps$.
\end{theorem}
\begin{proof}
	Consider the nonlinear operator $K$
	from $\mathds{R} \times L^\infty(\Omega; \mathds{R}^d)$ to $L^\infty(\Omega; \mathds{R}^d)$
	given by
	\[
		(K(\lambda, u))_i \coloneqq (-\Delta_{p_i})^{-1} \bigl(\lambda g_i(u)\bigr).
	\]
	By Example~\ref{ex:pLaplace} the operator $K$ is compact. By Lemma~\ref{lem:lipschitz}
	and local Lipschitz continuity of $g$ the operator $K$ is locally Lipschitz continuous. Moreover,
	\begin{align*}
		\|K(\lambda, u) - K(\lambda, \tilde{u})\|_\infty
			& \le c_1 \lambda^{\frac{1}{p-1}} \bigl( \|g(u)\|_\infty + \|g(\tilde{u})\|_\infty \bigr) \|g(u) - g(\tilde{u})\|_\infty \\
			& \le c_2 \lambda^{\frac{1}{p-1}} \|u - \tilde{u}\|_\infty
	\end{align*}
	by Lemma~\ref{lem:lipschitz} with constants $c_1$ and $c_2$ that depend only on the $p_i$, an upper bound $R$
	for $u$ and $\tilde{u}$ in $L^\infty(\Omega;\mathds{R}^d)$ and the Lipschitz constant of $g$ on $B_{\mathds{R}^d}(0,R)$.

	Hence the assumptions of Proposition~\ref{prop:curve} are satisfied. We deduce that
	the equation $K(\lambda,u) = u$, i.e., problem~\eqref{eq:system}, is locally solved by an implicit function
	$u_\lambda \coloneqq \phi(\lambda)$, that there are no other $L^\infty$-small solutions for small $\lambda$ and
	that the dependence of $u_\lambda$ on $\lambda$ is Lipschitz continuous with respect to the norm
	of $L^\infty(\Omega)$. Lipschitz continuous dependence
	in the norm of $W^{1,p}_0(\Omega)$ follows from the identity $u_\lambda = (-\Delta_p)^{-1} (\lambda g(u_\lambda))$
	since $(-\Delta_p)^{-1}$ is locally Lipschitz continuous from $W^{-1,p'}(\Omega)$ to $W^{1,p}_0(\Omega)$,
	see~\cite[Theorem~3.3.18 and Example~3.3.23]{Diss}.
\end{proof}

We have restricted ourselves to Dirichlet boundary conditions for simplicity.
Still, essentially the same arguments apply to the system
\[
	\left\{ \begin{aligned}
		|u_i|^{p_i-2}u_i - \Delta_{p_i} u_i & = \lambda g_i(u_1,\dots,u_d) && \text{on } \Omega \\
		|\nabla u_i|^{p-2} \frac{\partial u_i}{\partial\nu} & = 0 && \text{on } \partial\Omega, \\
	\end{aligned} \right.
\]
subject to Neumann boundary conditions if $\Omega \subset \mathds{R}^N$ is a bounded Lipschitz domain.
Moreover, we could let $g$ depend on $x \in \Omega$ or consider more general quasi-linear operators for which
the assumptions of~\cite[Theorem~3.1]{Bie10} hold.

\begin{remark}
	Theorem~\ref{thm:curve} generalizes the first step of the argument in~\cite{CR75} from $p=2$
	to a larger range. It would now
	be interesting to study the behavior of $u_\lambda$ as $\lambda$ increases.
	For $d=1$ and $g = \exp$, i.e., Gelfand's equation with arbitrary $p \in (1,\infty)$,
	it is known that there are no solutions for sufficiently large $\lambda$,
	see~\cite{AAP94}. One would suspect that one observes
	the same turning point structure as described in~\cite{CR75} for $p=2$
	where the curve ceases to exist.
	In fact, this is true if $\Omega$ is a ball~\cite[Example~3.1]{JS02},
	but it is not obvious how to handle the case $p \neq 2$ for general domains.
\end{remark}

\bibliographystyle{amsalpha}
\bibliography{pgelfand}

\providecommand{\bysame}{\leavevmode\hbox to3em{\hrulefill}\thinspace}
\providecommand{\MR}{\relax\ifhmode\unskip\space\fi MR }
\providecommand{\MRhref}[2]{%
  \href{http://www.ams.org/mathscinet-getitem?mr=#1}{#2}
}
\providecommand{\href}[2]{#2}
\begin{thebibliography}{AMd{\'O}M05}

\bibitem[AAP94]{AAP94}
J.~Garcia Azorero, I.~Peral Alonso, and J.P. Puel, \emph{Quasilinear problems
  with exponential growth in the reaction term}, Nonlinear Analysis: Theory,
  Methods \& Applications \textbf{22} (1994), no.~4, 481--498.

\bibitem[AC02]{AC02}
C{\'e}line Azizieh and Philippe Cl{\'e}ment, \emph{A priori estimates and
  continuation methods for positive solutions of {$p$}-{L}aplace equations}, J.
  Differential Equations \textbf{179} (2002), no.~1, 213--245.

\bibitem[AD07]{AD07}
Wolfgang Arendt and Daniel Daners, \emph{Uniform convergence for elliptic
  problems on varying domains}, Math. Nachr. \textbf{280} (2007), no.~1-2,
  28--49.

\bibitem[Ama76]{Amann76}
Herbert Amann, \emph{Fixed point equations and nonlinear eigenvalue problems in
  ordered {B}anach spaces}, SIAM Rev. \textbf{18} (1976), no.~4, 620--709.

\bibitem[AMd{\'O}M05]{AMM05}
Emerson A.~M. Abreu, Jo{\~a}o Marcos~do {\'O}, and Everaldo~S. Medeiros,
  \emph{Multiplicity of positive solutions for a class of quasilinear
  nonhomogeneous {N}eumann problems}, Nonlinear Anal. \textbf{60} (2005),
  no.~8, 1443--1471.

\bibitem[Ber77]{Berger}
Melvin~S. Berger, \emph{Nonlinearity and functional analysis}, Academic Press
  [Harcourt Brace Jovanovich Publishers], New York, 1977, Lectures on nonlinear
  problems in mathematical analysis, Pure and Applied Mathematics.

\bibitem[Bie10]{Bie10}
Markus Biegert, \emph{A priori estimates for the difference of solutions to
  quasi-linear elliptic equations}, Manuscripta Math. \textbf{133} (2010),
  no.~3-4, 273--306.

\bibitem[Br{\'e}73]{Brezis}
H.~Br{\'e}zis, \emph{Op\'erateurs maximaux monotones et semi-groupes de
  contractions dans les espaces de {H}ilbert}, North-Holland Publishing Co.,
  Amsterdam, 1973, North-Holland Mathematics Studies, No. 5. Notas de
  Matem{\'a}tica (50).

\bibitem[Bro69]{Br69}
F{\'e}lix~E. Browder, \emph{Remarks on nonlinear interpolation in {B}anach
  spaces}, J. Functional Analysis \textbf{4} (1969), 390--403.

\bibitem[CG03]{CG03}
Fabio Cipriani and Gabriele Grillo, \emph{Nonlinear {M}arkov semigroups,
  nonlinear {D}irichlet forms and applications to minimal surfaces}, J. Reine
  Angew. Math. \textbf{562} (2003), 201--235.

\bibitem[Cha57]{Ch57}
S.~Chandrasekhar, \emph{An introduction to the study of stellar structure},
  Dover Publications Inc., New York, N. Y., 1957. \MR{0092663 (19,1142b)}

\bibitem[CR73]{CR73}
Michael~G. Crandall and Paul~H. Rabinowitz, \emph{Bifurcation, perturbation of
  simple eigenvalues and linearized stability}, Arch. Rational Mech. Anal.
  \textbf{52} (1973), 161--180.

\bibitem[CR75]{CR75}
\bysame, \emph{Some continuation and variational methods for positive solutions
  of nonlinear elliptic eigenvalue problems}, Arch. Rational Mech. Anal.
  \textbf{58} (1975), no.~3, 207--218.

\bibitem[CS07]{CS07}
Xavier Cabr{\'e} and Manel Sanch{\'o}n, \emph{Semi-stable and extremal
  solutions of reaction equations involving the {$p$}-{L}aplacian}, Commun.
  Pure Appl. Anal. \textbf{6} (2007), no.~1, 43--67.

\bibitem[Dan88]{Dancer88}
E.~N. Dancer, \emph{The effect of domain shape on the number of positive
  solutions of certain nonlinear equations}, J. Differential Equations
  \textbf{74} (1988), no.~1, 120--156.

\bibitem[DD09]{DD09}
Daniel Daners and Pavel Dr{\'a}bek, \emph{A priori estimates for a class of
  quasi-linear elliptic equations}, Trans. Amer. Math. Soc. \textbf{361}
  (2009), no.~12, 6475--6500.

\bibitem[DH01]{DH01}
Pavel Dr{\'a}bek and Jes{\'u}s Hern{\'a}ndez, \emph{Existence and uniqueness of
  positive solutions for some quasilinear elliptic problems}, Nonlinear Anal.
  \textbf{44} (2001), no.~2, Ser. A: Theory Methods, 189--204.

\bibitem[ILU10]{ILU10}
Leonelo Iturriaga, Sebasti{\'a}n Lorca, and Pedro Ubilla, \emph{A quasilinear
  problem without the {A}mbrosetti-{R}abinowitz-type condition}, Proc. Roy.
  Soc. Edinburgh Sect. A \textbf{140} (2010), no.~2, 391--398.

\bibitem[JS02]{JS02}
Jon Jacobsen and Klaus Schmitt, \emph{The {L}iouville-{B}ratu-{G}elfand problem
  for radial operators}, J. Differential Equations \textbf{184} (2002), no.~1,
  283--298.

\bibitem[KL03]{KL03}
Dimitrios~A. Kandilakis and Athanasios~N. Lyberopoulos, \emph{Multiplicity of
  positive solutions for some quasilinear {D}irichlet problems on bounded
  domains in {${\mathbb R}^n$}}, Comment. Math. Univ. Carolin. \textbf{44}
  (2003), no.~4, 645--658.

\bibitem[Nit06]{Dipl}
R.~Nittka, \emph{The number of solutions of non-linear elliptic equations on
  varying domains}, Master's thesis, University of Ulm, 2006.

\bibitem[Nit10]{Diss}
\bysame, \emph{Elliptic and parabolic problems with {R}obin boundary conditions
  on {L}ipschitz domains}, Ph.D. thesis, University of Ulm, 2010.

\bibitem[PW02]{PM02}
Michael Plum and Christian Wieners, \emph{New solutions of the {G}elfand
  problem}, J. Math. Anal. Appl. \textbf{269} (2002), no.~2, 588--606.

\bibitem[Sho97]{Showalter}
R.~E. Showalter, \emph{Monotone operators in {B}anach space and nonlinear
  partial differential equations}, Mathematical Surveys and Monographs,
  vol.~49, American Mathematical Society, Providence, RI, 1997.

\bibitem[Wue08]{Wuertz}
M.~Wuertz, \emph{The implicit function theorem for {L}ipschitz functions and
  applications}, Master's thesis, University of Missouri, 2008.

\end{thebibliography}

\end{document}